\documentclass[12pt,twoside]{amsart}
\usepackage{amsmath, amssymb, amscd,tabularx,amsmath, verbatim, amsthm, amssymb, mathrsfs, times,latexsym,amscd,graphicx, color, hyperref}
\usepackage{cite}

\usepackage[all]{xy}

\evensidemargin \oddsidemargin
\newtheorem{thm}{Theorem}[section]

\newtheorem{lem}[thm]{Lemma}
\newtheorem{lemma}[thm]{Lemma}

\newtheorem{proposition}[thm]{Proposition}
\newtheorem{prop}[thm]{Proposition}

\theoremstyle{definition}

\theoremstyle{remark}
\newtheorem{remark}[thm]{Remark}

\newcommand{\ds}{\displaystyle}

\title{Polynomial Identities on Eigenforms}

\author {Joseph Richey}
\address{Department of Mathematics\\ 
530 Church Street\\
University of Michigan\\
Ann Arbor, MI 48109}
\email[] {josephlr@umich.edu}

\author {Noah Shutty}
\address{Department of Mathematics\\ 
530 Church Street\\
University of Michigan\\
Ann Arbor, MI 48109}
\email[] {noajshu@umich.edu}
\urladdr[]{\url{umich.edu/~noajshu}}

\newcommand{\N}{\mathbb N}

\newcommand{\C}{\mathbb C}
\newcommand{\Z}{\mathbb{Z}}

\DeclareMathOperator{\SL}{SL}

\newcommand{\M}{{\text{M}}}

\parskip=\medskipamount

\begin{document}
\date{\today}

\begin{abstract}
In this paper, we fix a polynomial with complex coefficients and determine the eigenforms for $\SL_2\left(\Z\right)$ which can be expressed as the fixed polynomial evaluated at other eigenforms. In particular, we show that when one excludes trivial cases, only finitely many such identities hold for a fixed polynomial.
\end{abstract}

\maketitle

\section{Introduction}\label{section:intro}
Identities between Hecke eigenforms often give rise to surprising relationships between arithmetic functions. A well-known example involves the sum of divisor functions $\sigma_3(n)$ and $\sigma_7(n)$:

\begin{equation}\label{classical_identity}
\sigma_7(n)=\sigma_3(n)+120\sum_{j=1}^{n-1}{\sigma_3(j)\sigma_3(n-j)}.
\end{equation}

This identity is easily derived from the fact that $E_4E_4=E_8$, where $E_4$ (respectively $E_8$) is the weight 4 (respectively 8) Eisenstein series for $\SL_2(\Z)$. Any product relation among eigenforms gives rise to a similar identity. Duke \cite{Duke} has shown that an eigenform for $\SL_2(\Z)$ may be decomposed as a product of two others in only sixteen cases (independently observed by Ghate \cite{Ghate_initial}). Ghate \cite{Ghate_final} later considered eigenforms of higher level, and showed that there are still only finitely many such decompositions as long as the level is squarefree, and the weights of all eigenforms are at least 3. Johnson \cite{Johnson} has recently extended this result by showing that only a finite number of decompositions involving eigenforms of weight at least 2 exist for any given level and character.

Emmons and Lanphier \cite{Emmons} considered product decompositions involving any number of eigenforms for $\SL_2(\Z)$, and showed that the only relations that arise are the 16 identified by Duke and Ghate, and a few trivially implied by them. In fact, these results show that the product decompositions they describe occur only when dimension considerations require it.

In this paper we move from monomial to polynomial decompositions. Our main result is the following theorem:

\begin{thm}\label{thm:main}
For a fixed $P \in \C[x_1,x_2,\dots,x_n]$ there exist only finitely many $n+1$ tuples $(f_1,f_2,\dots,f_n,h)$ of non-zero eigenforms for $\SL_2(\Z)$ such that 
\begin{equation}\label{eq:main}
P(f_1,f_2,\dots,f_n)=h,
\end{equation}
where the weights of all the $f_i$ are less than the weight of $h$.
\end{thm}
Note that in order to discuss polynomial relationships amongst eigenforms, the addition of these eigenforms must be defined.  Thus, if we have a polynomial relation on eigenforms, each term of the polynomial will have the same weight. 

Out proof considers a particular decomposition of an arbitrary eigenform, and obtains an upper bound for the weight of that form. For the Eisenstein series, we rely on the fact that we have an explicit formula for their Fourier coefficients. The cuspidal case is more difficult. Our proof relies on several number theoretic lemmas, the Hecke relation satisfied by the Fourier coefficients of eigenforms, and bounds on the magnitude of the Fourier coefficients of cuspidal eigenforms.

As mentioned above, it was observed by Duke \cite{Duke}, Ghate \cite{Ghate_final}, and Emmons and Lanphier \cite{Emmons} that the only product decompositions of eigenforms over $\SL_2(\Z)$ are those forced by dimension considerations. Extensive computations suggest that this is true for most, if not all, polynomial decompositions as well.

\section{Notation and Conventions}\label{section:notation}

Throughout this paper, if $f$ is a modular form over $\SL_2(\Z)$, we will let $a_f(n)$ denote the $n^{th}$ Fourier coefficient of $f$. In other words, $f$ has a Fourier expansion given by: 
\begin{equation}\label{fourier-def}
f(z) = \sum_{n=0}^{\infty}a_f(n)q^n,
\end{equation}
where $q=e^{2\pi i z}$.   

All Eisenstein series will be normalized so as to have constant coefficient equal to one. Thus, if $E_k$ is the unique Eisenstein series of weight $k$ (where $k$ is an even integer with $k\geq 4$), then
\begin{equation}\label{eq:Eisenstein-def}
E_{k}(z)=1+C_{k}\sum_{n=1}^{\infty}\sigma_{k-1}(n)\, q^{n},
\end{equation}
where $C_k=\frac{(2\pi i)^{k}}{(k-1)!\zeta(k)}=\frac{-2k}{B_{k}}$
(here $B_{k}$ is the $k^{th}$ Bernoulli number).

In this paper, we will prove the following theorem which implies Theorem~\ref{thm:main}.

\begin{thm}\label{thm:equivmain}
Fix a positive integer $r$ and for each $1 \leq i \leq r$ fix $A_{i}\in\mathbb{C}$
and $n_{i},m_{i}\in\mathbb{N}$ such that $n_i+m_i\geq 2$. There are only finitely many eigenforms h such that
\[
\sum_{i=1}^{r}A_{i}P_{i}F_{i}=h,
\]

where $P_{i}$ is the product of $n_{i}$ Eisenstein series and $F_{i}$ is the product of $m_{i}$ cuspidal eigenforms.

\end{thm}

To see why Theorem~\ref{thm:equivmain} implies Theorem~\ref{thm:main}, note that a fixed $P \in \C[x_1,x_2,\dots,x_n]$, of finite degree may only have finitely many distinct terms. For each $f_i$ in \eqref{eq:main}, choose whether it will be an Eisenstein series or a cuspidal eigenform. Clearly, there are only $2^n$ ways to preform this choice for all of the $f_i$.  Then, we can write each of the $r$ terms as $APF$, where $A \in\mathbb{C}$, $P$ is any product of $m_P$ Eisenstein series, and $F$ is any product of $m_F$ cuspidal eigenforms.  In this form, the requirement in Theorem~\ref{thm:main} that the weights of the $f_i$ be less than the weight of $h$ is equivalent to the requirement in Theorem~\ref{thm:equivmain} that $m_P+m_F \geq 2$. The equivalence follows from the finiteness of choices for the $f_i$.

To prove Theorem~\ref{thm:equivmain} it is sufficient to bound the weight of $h$, as there are only finitely many eigenforms of a given weight. 

In Section~\ref{proof:Eisenstein}, we show that if 
\begin{equation}\label{eq:Eisenstein-poly}
\sum_{i=1}^{r}A_{i}P_{i}F_{i}=E_\ell,
\end{equation}
then $\ell$ is bounded (depending only on $r,A_{i},n_{i},m_{i}$).
In Section~\ref{proof:cuspidal}, we show that if
\begin{equation}\label{eq:cuspidal-poly}
\sum_{i=1}^{r}A_{i}P_{i}F_{i}=h,
\end{equation}
where $h$ is a cuspidal eigenform, then the weight of $h$ is bounded (again depending only on $r,A_{i},n_{i},m_{i}$).
As all eigenforms are either Eisenstein series or cusp forms, the proof of Theorem~\ref{thm:equivmain} in cases \eqref{eq:Eisenstein-poly} and \eqref{eq:cuspidal-poly} suffices to prove Theorem~\ref{thm:main}. 

\section{Preliminaries}\label{section:prelim}

Here we prove some preliminary results about the properties of the Fourier coefficients of modular forms, as well as some basic number theoretic lemmas, which will be necessary in the proof of Theorem~\ref{thm:equivmain}.
\subsection{Number Theoretic Results}
\begin{lemma}\label{lem:expswitch}
If $a$ and $b$ are integers satisfying $a\geq b>1$, then the following equation holds:
\begin{equation*}
(a+1)^{a+1}(b-1)^{b-1} > a^a b^b.
\end{equation*}
\end{lemma}
\begin{proof}
First, fix any integers $a$ and $b$ such that $a \geq b>1$. Now consider the function
\[
f(x) := (a+x)^{a+x}(b-x)^{b-x}.
\]
Note that this function is differentiable on $[0,1]$  with derivative
\[
f^\prime(x) = f(x)\left( \log(a+x) - \log(b-x) \right)
\]
and that $f^\prime(x)$ is positive on $(0,1]$, which implies $f(x)$ is increasing on $(0,1]$. Thus,
\[
a^a b^b = f(0) < f(1) = (a+1)^{a+1}(b-1)^{b-1}.
\]
\end{proof}

The definition of $C_k$ in equation~\eqref{eq:Eisenstein-def} and the trivial bounds $\ds{1<\zeta(k)\leq\frac{k}{k-1}}$, which hold whenever $k>1$, imply the folowing.
\begin{lemma}\label{lem:CkRatio}
Let $C_k$ be defined as in equation~\eqref{eq:Eisenstein-def}. For even integers $k\geq 4$, 
\[
\lim_{k\rightarrow\infty}\frac{\left|C_{k+2}\right|}{\left|C_{k}\right|}=0.
\]
\end{lemma}

\begin{prop}\label{prop:sum=constant}
Let $\left\{ D_{i}\right\} _{i\in\mathbb{N}}$ be a sequence of non-zero
complex numbers such that
\[
\lim_{i\rightarrow\infty}\frac{\left|D_{i+1}\right|}{\left|D_{i}\right|}=0.
\]
Then, for fixed $A_{1},A_{2},\cdots A_{m}\in\mathbb{C}\setminus\left\{ 0\right\} ,$
there are only finitely many tuples $\left(k_{1},k_{2},\cdots,k_{m}\right)$
of integers with $1\leq k_{1}<k_{2}<\cdots<k_{m}$ such that 
\[
A_{1}D_{k_{1}}+A_{2}D_{k_{2}}+\cdots+A_{m}D_{k_{m}}=s
\]
for any fixed $s\in\mathbb{C}.$
\end{prop}
\begin{proof}
Observe that the sequence $\left|D_{i}\right|$ is eventually strictly decreasing and $\lim_{i\rightarrow\infty}\left|D_{i}\right|=0$. We prove the theorem by induction on $m.$ If $m=1,$ we note that as the sequence $\left|D_{i}\right|$ is eventually strictly decreasing, the equation $D_{j}=c$ can only be satisfied by finitely many $j$ for a fixed constant $c\in\mathbb{C}.$ Thus, there are only finitely many $k_{1}$ such that $A_{1}D_{k_{1}}=s.$ The proposition for $m=1$ follows.

If $m>1,$ we will first show that, for fixed $A_{1},A_{2},\cdots A_{m}$
and $s,$ $k_{1}$ is bounded.

Suppose first that $s=0$. Let $A:=\min\left(1,\frac{\left|A_{1}\right|}{\sum_{j=2}^{m}\left|A_{j}\right|}\right).$ As $\lim_{i\rightarrow\infty}\frac{\left|D_{i+1}\right|}{\left|D_{i}\right|}=0,$
the exists some $M$ such that for all $i>M$ we have $\frac{\left|D_{i+1}\right|}{\left|D_{i}\right|}<A.$
As $A\leq1,$ this implies that for all $j>i>M$ we have $\frac{\left|D_{j}\right|}{\left|D_{i}\right|}<A.$ 
Applying the triangle inequality to $A_{1}D_{k_{1}}+A_{2}D_{k_{2}}+\cdots+A_{m}D_{k_{m}}=0$ and simplifying gives,
\[
1\leq\left|\frac{A_{2}}{A_{1}}\right|\left|\frac{D_{k_{2}}}{D_{k_{1}}}\right|+\left|\frac{A_{3}}{A_{1}}\right|\left|\frac{D_{k_{3}}}{D_{k_{1}}}\right|+\cdots+\left|\frac{A_{m}}{A_{1}}\right|\left|\frac{D_{k_{m}}}{D_{k_{1}}}\right|.
\]
Assume $k_1>M$.  As $k_{i}>k_{1}>M$ for all $2\leq i \leq m,$  we
have, by the definition of $A$,
\[
1<\left|\frac{A_{2}}{A_{1}}\right|A+\left|\frac{A_{3}}{A_{1}}\right|A+\cdots+\left|\frac{A_{m}}{A_{1}}\right|A=\left(\frac{\sum_{j=2}^{m}\left|A_{j}\right|}{\left|A_{1}\right|}\right)A\leq1.
\]
This is a contradiction, so $k_{1}$ is bounded by $M.$

Now suppose that $s\neq 0$. As $\lim_{i\rightarrow\infty}\left|D_{i}\right|=0,$ there exists
an $N$ such that for all $i>N$ we have $\left|D_{i}\right|<\frac{\left|s\right|}{\sum_{j=1}^{m}\left|A_{j}\right|}.$ 
Applying the triangle inequality to $A_{1}D_{k_{1}}+A_{2}D_{k_{2}}+\cdots+A_{m}D_{k_{m}}=s$ gives,
\[
\left|s\right|\leq\left|A_{1}\right|\left|D_{k_{1}}\right|+\left|A_{2}\right|\left|D_{k_{2}}\right|+\cdots+\left|A_{m}\right|\left|D_{k_{m}}\right|.
\]
Assume $k_1>N$. As $k_{i}\geq k_{1}>N$ for all $1\leq i\leq m,$
we have 
\[
\left|s\right|<\left(\left|A_{1}\right|+\left|A_{2}\right|+\cdots+\left|A_{m}\right|\right)\frac{\left|s\right|}{\sum_{j=1}^{m}\left|A_{j}\right|}=\left|s\right|,
\]
a contradiction. Thus, $k_{1}$ is bounded by $N.$

As $k_{1}$ is bounded, for the equation $A_{1}D_{k_{1}}+A_{2}D_{k_{2}}+\cdots+A_{m}D_{k_{m}}=s$
to be satisfied, one of the finitely many equations of the form $A_{2}D_{k_{2}}+\cdots+A_{m}D_{k_{m}}=s-A_{1}D_{c}$
(where $c$ is one of the finitely many possible values for $k_{1})$
must be satisfied. However, by the inductive hypothesis, we see that each of
these equations only admits finitely many solutions, so we have finiteness
in general. The proposition follows by induction.
\end{proof}

\begin{lem}\label{lem:divisorgrowth}
If $d(n)$ is the number of divisors of $n$, for any positive integer $n$
\[
d(n)\leq2\sqrt{n}.
\]
\end{lem}
\begin{proof}
First, note that $n_0$ is a divisor of $n$ if and only if $\frac{n}{n_0}$ is a divisor of $n$. Thus, $d(n)$, the number of divisors of $n$, is bounded above by twice the number of divisors less than or equal to $\sqrt{n}$.  However, the number of divisors less than or equal to $\sqrt{n}$ is bounded above by $\sqrt{n}$. The lemma follows.
\end{proof}

\subsection{Bounds on Fourier coefficients}
\begin{prop}\label{prop:boundproductofEisenstein}
Fix positive integers $n,m$ and let $P$ be any product of $m$ Eisenstein series. Then there exists a real number $B(n,m)$, depending only on $n$ and $m$, such that
\[
\left|a_P(n)\right|\leq B(n,m).
\]
\end{prop}
\begin{proof}
We proceed by induction on $m$. If $m=1$, then $P=E_k$ for some $k\geq 4$ and 
\[
\left|a_P(n)\right|=\left|C_k\right|\sigma_{k-1}(n)=\frac{(2\pi)^k}{(k-1)!\zeta(k)}\sigma_{k-1}(n).
\]
Fix $n$. As
\[
\sigma_{k-1}(n)=\sum_{d|n}d^{k-1}\leq\sum_{d=1}^{n}d^{k-1}\leq\sum_{d=1}^{n}n^{k-1}=n^k
\]
and $\zeta(k) > 1$, we have
\[
\left|a_P(n)\right|\leq\frac{(2\pi n)^k}{(k-1)!}.
\]
However,
\[
\lim_{k\rightarrow\infty}\frac{(2\pi n)^k}{(k-1)!}=0,
\] 
so the sequence (as a function of $k$) is bounded. Therefore, there exists a real number $B$ depending only on $n$ such that $B > \frac{(2\pi n)^k}{(k-1)!}$. Letting $B=B(n,1)$ gives the proposition for $m=1$.

Now we assume the proposition is true for all $m\leq r$. So if $m=r+1$, then $P=E_k P^\prime$, where $P^\prime$ is a product of $r$ Eisenstein series. Now we have
\[
\left|a_P(n)\right|=\left|a_{E_k P^\prime}(n)\right|=\left|\sum_{i=0}^n a_{E_k}(i)a_{P^\prime}(n-i)\right|\leq\sum_{i=0}^n \left| a_{E_k}(i)\right| \left| a_{P^\prime}(n-i)\right|.
\]
Using the inductive hypothesis gives
\[\left|a_P(n)\right|\leq\sum_{i=0}^n B(i,1)B(n-i,r).\]
Letting $B(n,r+1)=\sum_{i=0}^n B(i,1)B(n-i,r)$ completes the proof.
\end{proof}

\begin{remark}
Deligne \cite{deligne74} has shown that if $f$ is a cuspidal eigenform of weight $\ell$, then
\[
a_f(n)\leq d(n)n^{\frac{\ell-1}{2}}
\]
for all $n\in\N$ (see also Lemma 0.0.0.3 of \cite{conrad}). Using Lemma~\ref{lem:divisorgrowth}, we see that
\begin{equation}\label{ram-pet}
\left|a_f(n)\right|\leq 2n^{\frac{\ell}{2}}
\end{equation}
for all $n\in\N$.
\end{remark}
\begin{lemma}\label{lem:absboundcusp}
Fix positive integers $n,m$ and let $F$ be any product of $m$ cuspidal eigenforms. If $\ell$ denotes the weight of F, then
\[
\left|a_F(n)\right|\leq 2^m n^{\frac{\ell}{2}+m-1}.
\]
\end{lemma}
\begin{proof}
We prove this lemma by induction on $m$. If $m=1$, then the lemma follows from \eqref{ram-pet}. Now assume the lemma holds for $m=r$ and fix $n$. Now let $F$ be a product of $r+1$ cuspidal eigenforms. This means $F=Gf$, where $f$ is a cuspidal eigenform and $G$ is a product of $r$ cuspidal eigenforms. We will let $\ell$ denote the weight of $F$ and let $\ell^\prime$ denote the weight of $G$. Using the inductive hypothesis and the triangle inequality, we have
\begin{align*}
\left|a_{F}(n)\right| & \leq\sum_{i=1}^{n-1}\left|a_{G}(i)\right|\left|a_{f}(n-i)\right| \leq\sum_{i=1}^{n-1}2^r i^{\frac{\ell^\prime}{2}+r-1}2\left(n-i\right)^{\frac{\ell-\ell^\prime}{2}}\\
& \leq\sum_{i=1}^{n-1}2^{r} n^{\frac{\ell^\prime}{2}+r-1}2n^{\frac{\ell-\ell^\prime}{2}} = \sum_{i=1}^{n-1}2^{r+1} n^{\frac{\ell}{2}+r-1} = (n-1)2^{r+1} n^{\frac{\ell}{2}+r-1}\\
& \leq 2^{r+1} n^{\frac{\ell}{2}+(r+1)-1}.
\end{align*}
This proves the lemma when $m=r+1$; the full lemma follows by induction.
\end{proof}
\begin{prop}\label{prop:boundproductofcusp}
Fix positive integers $n$ and $m$ with $m\geq2$ and let $F$ be any product of $m$ cuspidal eigenforms. Let $\ell$ denote the weight of $F$. Then for every $k\geq 0$ there exists a positive integer $L = L(n,m,k)$ such that if $\ell > L$, then
\[
\left|a_F(n)\right|\leq n^{\frac{\ell}{2}-k}.
\]
\end{prop}
\begin{proof}
We proceed by induction on $m$. Let $m=2$ and fix $n$ and $k$, then $F=fg$, where $f,g$ are cuspidal eigenforms with weights $a,\ell-a$.  Using that, \eqref{ram-pet}, and the triangle inequality, we have
\begin{align*}
\left|a_{F}(n)\right| & \leq\sum_{i=1}^{n-1}\left|a_{f}(i)\right|\left|a_{g}(n-i)\right|\\
& \leq\sum_{i=1}^{n-1}2i^{\frac{a}{2}}2\left(n-i\right)^{\frac{\ell-a}{2}}\\
& \leq4(n-1)\max_{0\leq i\leq n}\left(i^{\frac{a}{2}}\left(n-i\right)^{\frac{\ell-a}{2}}\right).
\end{align*}

Using techniques of differential calculus, one can easily verify that
when $\ell$ and $a$ are held constant, $i^{\frac{a}{2}}\left(n-i\right)^{\frac{\ell-a}{2}}$
is maximized (viewed as a function of a real variable) when $i=\frac{na}{\ell}$ and
$n-i=\frac{n\left(\ell-a\right)}{\ell}$. Thus,
\begin{align*}
\left|a_{F}(n)\right| & \leq4(n-1)\left(\frac{na}{\ell}\right)^{\frac{a}{2}}\left(\frac{n\left(\ell-a\right)}{\ell}\right)^{\frac{\ell-a}{2}}\\
& =4(n-1)n^{\frac{a}{2}}n^{\frac{\ell-a}{2}}\left(\frac{a}{\ell}\right)^{\frac{a}{2}}\left(\frac{\ell-a}{\ell}\right)^{\frac{\ell-a}{2}}=4(n-1)n^{\frac{\ell}{2}}\left(\frac{a}{\ell}\right)^{\frac{a}{2}}\left(\frac{\ell-a}{\ell}\right)^{\frac{\ell-a}{2}}\\
& \leq4(n-1)n^{\frac{\ell}{2}}\max_{12\leq x\leq\ell-12}\left(\left(\frac{x}{\ell}\right)^{\frac{x}{2}}\left(\frac{\ell-x}{\ell}\right)^{\frac{\ell-x}{2}}\right),
\end{align*}
where the last inequality follows from the fact that $12\leq a$
and $12\leq\ell-a$ because all cuspidal modular forms have
weight 12 or greater. Again using the tools of differential calculus,
we see that when $\ell$ is held constant, $\left(\frac{x}{\ell}\right)^{\frac{x}{2}}\left(\frac{\ell-x}{\ell}\right)^{\frac{\ell-x}{2}}$
is maximized over $\left[12,\ell-12\right]$ when $x=12$ or $x=\ell-12$.
Thus,
\[
\left|a_{F}(n)\right|\leq4(n-1)n^{\frac{\ell}{2}}\left(\frac{12}{\ell}\right)^{6}\left(\frac{\ell-12}{\ell}\right)^{\frac{\ell-12}{2}}<4(n-1)n^{\frac{\ell}{2}}\left(\frac{12}{\ell}\right)^{6}
\]
because $\frac{\ell-12}{\ell}<1$. Now, we let $L(n,2,k)=12\sqrt[6]{4(n-1)n^k}$. Thus, if $\ell>L(n,2,k)$, then
\[
\left|a_{F}(n)\right|\leq 4(n-1)n^{\frac{\ell}{2}}\frac{12^{6}}{\ell^{6}}<4(n-1)n^{\frac{\ell}{2}}\frac{12^{6}}{12^{6}4(n-1)n^k}=n^{\frac{\ell}{2}-k}.
\]
Therefore, the proposition is true when $m=2$.

Now we will assume the proposition holds for $m=r$ and fix $n$ and $k$. Let $F=f_1f_2\dots f_{r+1}$, where all of the $f_i$ are cuspidal eigenforms of weight $\ell_i$ and 
$$\ell > \max\left(\frac{r+1}{r}\max_{1\leq i < n}\left(L(i,r,0)\right),12\sqrt[6]{2n^k(n-1)}\right).$$
As multiplying modular forms adds their weights, we have
\[\ell=\ell_1+\ell_2+\dots+\ell_{r+1}.\]
As $\ell > \frac{r+1}{r}\max_{1\leq i < n}\left(L(i,r,0)\right)$, if $\ell_s=\min_{1\leq j\leq n}\left(\ell_j\right)$, then
\[
\ell_s=\min_{1\leq j\leq r+1}\left(\ell_j\right)\leq\frac{\ell}{r+1}<\frac{\max_{1\leq i < n}\left(L(i,r,0)\right)}{r},
\]
which implies
\[
\ell_{1}+\dots+\ell_{s-1}+\ell_{s+1}+\dots+\ell_{r+1}>\max_{1\leq i < n}\left(L(i,r,0)\right).
\]
Let $F^{\prime}=f_{1}\dots f_{s-1}f_{s+1}\dots f_{r+1}$ and note that
$F=F^{\prime}f_{s}$.

Expanding the Fourier expressions of $F,F^{\prime},f_{s}$ and
multiplying gives
\[
a_{F}(n)=\sum_{i=1}^{n-1}a_{F^{\prime}}(i)a_{f_{s}}(n-i).
\]
As the weight of $F^{\prime}$ is greater than $L(i,r,0)$ for $1\leq i <n$, the inductive hypothesis applies and
\[
\left|a_{F^\prime}(i)\right|\leq n^{\frac{\ell-\ell_s}{2}}
\]
for $1\leq i <n$.
Using that, \eqref{ram-pet}, and the triangle inequality, we have
\begin{align*}
\left|a_{F}(n)\right| & \leq\sum_{i=1}^{n-1}\left|a_{F^{\prime}}(i)\right|\left|a_{f_{s}}(n-i)\right|\\
& \leq\sum_{i=1}^{n-1}i^{\frac{\ell-\ell_{s}}{2}}2\left(n-i\right)^{\frac{\ell_{s}}{2}}\\
& \leq2(n-1)n^{\frac{\ell}{2}}\left(\frac{12}{\ell}\right)^{6}.
\end{align*}
Here, the last inequality is derived using identical techniques to those above. As $\ell>12\sqrt[6]{2n^k(n-1)}$, we have
\[
\left|a_{F}(n)\right|\leq2(n-1)n^{\frac{\ell}{2}}\frac{12^{6}}{\ell^{6}}<2(n-1)n^{\frac{\ell}{2}}\frac{12^{6}}{12^{6}2n^k(n-1)}=n^{\frac{\ell}{2}-k}.
\]
Thus, the proposition is true when $m=r+1$ because we can let \[
L(n,r+1,k)=\max\left(\frac{r+1}{r}\max_{1\leq i < n}\left(L(i,r,0)\right),12\sqrt[6]{2n^k(n-1)}\right).
\] 
The full proposition follows by induction. 
\end{proof}

\section{Proof of Theorem~\ref{thm:equivmain} when $h$ is an Eisenstein series}\label{proof:Eisenstein}

Recall that we denote an Eisenstein series of weight $\ell$ by $E_\ell$, as in \eqref{eq:Eisenstein-def}. Suppose we have the decomposition

\begin{equation}
\sum_{i=1}^{r}A_{i}P_{i}F_{i}=E_\ell, 
\end{equation}
where $A_{i}\in\C$ and $n_{i},m_{i}\in\N$ are fixed and $P_{i}$ is the product of $n_i$ Eisenstein series and $F_i$ is the product of $m_i$ cuspidal eigenforms. 

Note first that each term in the sum for which $m_i > 0$ is a cusp form, and hence has constant term $0$. Since the sum is an Eisenstein series with constant term 1, the set $B=\{ 1 \leq i \leq r \mid m_i = 0  \}$ must be nonempty. Now consider the $q$ term of each product. For $j \notin B$, the coefficient is either $A_j$ or $0$, depending on whether the product contains one or multiple cusp forms. For $j \in B$, the $q$ coefficient is the sum of the coefficients $C_{k}$ for each $E_k$ in $P_j$. Since $\left| B \right| \leq r$, there are a finite number of possibilities for $B$, and a finite number of associated possible choices of $m_j$ for each $j \notin B$ Therefore, for each polynomial, equality of the $q$ coefficients on each side implies one of a finite number of relations of the following form.
\[
\sum_{j \in B}{A_j \sum{}^{'}C_{k} } = s + C_{\ell}
\]
Here the inner sum is over the $k$ for each $E_k$ present in $P_j$, and $s$ is a complex number which depends on the choice of $m_j$ and $A_j$ for $j \notin B$. Regrouping, we find that if there are $r'$ distinct $C_k$ present in the relation, we obtain for certain $B_i\in \mathbb{C} $ that
\[
\sum_{i = 1}^{r'}{B_i C_{k_i}} = s'
\]
for some tuple $\left(k_1, \dots , k_{r'}  \right)$. In fact, since $r' \leq 1 + \sum_{j \in B}{m_j} $ (all possible inputs and single output), there are finitely many such relations. From Lemma~\ref{lem:CkRatio} and Proposition~\ref{prop:sum=constant}, each such equation only has finitely many solutions, and so we obtain finiteness of all solutions. Equivalently, there must be some finite weight $\ell_0$ associated with each polynomial such that if $\ell > \ell_0$ then there are no solutions to  \eqref{eq:Eisenstein-poly}.

\section{Proof of Theorem~\ref{thm:equivmain} when $h$ is a cuspidal eigenform}\label{proof:cuspidal}

In this section we consider equations of the form:
\begin{equation*}
\sum_{i=1}^{r}A_{i}P_{i}F_{i}=h, 
\end{equation*}
where $A_{i}\in\C$ and $n_{i},m_{i}\in\N$ are fixed and $P_{i}$ is the product of $n_i$ Eisenstein series and $F_i$ is the product of $m_i$ cuspidal eigenforms.
In order to prove Theorem~\ref{thm:equivmain} we first consider a single term of the above polynomial.
\begin{proposition}\label{prop:cuspidalcaseassist}
Fix positive integers $n,m_P,m_F$ with $n\geq3$ and $m_P+m_F\geq2$. Let $P$ be a product of $m_P$ Eisenstein series, $F$ be a product of $m_F$ cuspidal eigenforms, and $\ell$ be the weight of $PF$. If $\ell > M$, then
\[
\left|a_{PF}(n)\right| \leq n^{\frac{\ell}{2}-1},
\]
where $M$ depends only on $n,m_P,m_F$.
\end{proposition}
\begin{remark}
In essence, our proof of Proposition~\ref{prop:cuspidalcaseassist} follows from Propositions~\ref{prop:boundproductofEisenstein} and \ref{prop:boundproductofcusp}.  In Proposition~\ref{prop:boundproductofEisenstein}, we proved a bound for each Fourier coefficient of a product of Eisenstein series using the explicit formula for the Fourier coefficients of a single Eisenstein series. In Proposition~\ref{prop:boundproductofcusp}, we proved a bound for each Fourier coefficient of a product of cuspidal eigenforms using Deligne's bound on the Fourier coefficients of a cuspidal eigenform.
\end{remark}
\begin{proof}
We divide this into 5 cases. We will use the definitions of $L(n,m,\kappa)$ and $B(n,m)$ from Proposition~\ref{prop:boundproductofcusp} and Proposition~\ref{prop:boundproductofEisenstein}, respectively.

\emph{Case 1:} $m_P=0$

This implies $m_F\geq2$ and $P=1$. Thus, if $\ell > L(n,m_F,1)$, then by Proposition~\ref{prop:boundproductofcusp} 
\[
\left|a_{PF}(n)\right| = \left|a_{F}(n)\right| \leq n^{\frac{\ell}{2}-1}.
\]

\emph{Case 2:} $m_F=0$

This implies $m_P\geq2$ and $F=1$. Thus, if $\ell > 2\log_n\left(B(n,m_P)\right)+2$, then by Proposition~\ref{prop:boundproductofEisenstein} 
\[
\left|a_{PF}(n)\right| = \left|a_{P}(n)\right| \leq B(n,m_P) = n^{\frac{2\log_n\left(B(n,m_P)\right)+2}{2}-1}\leq n^{\frac{\ell}{2}-1}.
\]

\emph{Case 3:} $m_F=1$

This implies $m_P\geq1$ and that $F$ is a cuspidal eigenform. Using Proposition~\ref{prop:boundproductofEisenstein}, equation \eqref{ram-pet}, and the triangle inequality we have
\[
\left|a_{PF}(n)\right| \leq \sum_{i=0}^{n-1}\left|a_{P}(i)\right|\left|a_{F}(n-i)\right| \leq \left|a_P(0)\right|2n^\frac{\ell-4}{2}+\sum_{i=1}^{n-1} B\left(i,m_P\right)2(n-i)^{\frac{\ell-4}{2}},
\]
where the second inequality follows from the fact that the weight of $F$ is at most $\ell-4$ because $P$ has weight at least 4.
Simplifying,
\begin{align*}
\left|a_{PF}(n)\right| & \leq 2n^{\frac{\ell}{2}-2} + \sum_{i=1}^{n-1} B\left(i,m_P\right)2(n-i)^{\frac{\ell-4}{2}} \\
& \leq \frac{2}{3}n^{\frac{\ell}{2}-1} + (n-1)\max_{1\leq i < n}\left(B\left(i,m_P\right)\right)2(n-1)^{\frac{\ell-4}{2}} \\
& =\left(\frac{2}{3} + 2\max_{1\leq i < n}\left(B\left(i,m_P\right)\right)\left(\frac{n-1}{n}\right)^{\frac{\ell}{2}-1}\right)n^{\frac{\ell}{2}-1},
\end{align*}
where the first inequality follows from $\left|a_P(0)\right|=1$ and the second inequality follows from our assumption that $n \geq 3$. Now, if $\ell > 2\log_{\frac{n}{n-1}}\left(6\max_{1\leq i < n}\left(B\left(i,m_P\right)\right)\right)+2$, then 
\[
\left|a_{PF}(n)\right| \leq n^{\frac{\ell}{2}-1}.
\]
For the sake of notation let $\kappa\left(n,m_P\right)=\sqrt[n]{\max_{0\leq i < n}\left(B\left(i,m_P\right)\right)}$.

\emph{Case 4:} $m_F\geq2$ and weight of $F$ is greater than $L\left(n,m_F,\kappa\left(n,m_P\right)\right)$

This implies that
\[
\left|a_{PF}(n)\right| \leq \sum_{i=0}^{n-1}\left|a_{P}(i)\right|\left|a_{F}(n-i)\right| \leq n\max_{0\leq i < n}\left(B\left(i,m_P\right)\right)n^{\frac{\ell}{2}-\kappa\left(n,m_P\right)} \leq n^{\frac{\ell}{2}-1},
\]
where the second inequality follows from Proposition~\ref{prop:boundproductofcusp} and Proposition~\ref{prop:boundproductofEisenstein}, and the third inequality follows from the definition of $\kappa\left(n,m_P\right)$.

\emph{Case 5:} $m_F\geq2$ and weight of $F$ is less than or equal to $L\left(n,m_F,\kappa\left(n,m_P\right)\right)$

Let $\eta=L\left(n,m_F,\kappa\left(n,m_P\right)\right)$. Using Lemma~\ref{lem:absboundcusp} and Proposition~\ref{prop:boundproductofEisenstein}, we have
\[
\left|a_{PF}(n)\right| \leq \sum_{i=0}^{n-1}\left|a_{P}(i)\right|\left|a_{F}(n-i)\right| \leq n\max_{0\leq i < n}\left(B\left(i,m_P\right)\right)2^{m_F}n^{\frac{\eta}{2}+m_F-1}.
\]
However, as all cuspidal eigenforms have weight of at least 12, $m_F\leq\frac{\eta}{12}$. Thus,
\[
\left|a_{PF}(n)\right| \leq \max_{0\leq i < n}\left(B\left(i,m_P\right)\right) 2^{\frac{\eta}{12}}n^{\frac{7\eta}{12}}.
\]
If $\ell > \frac{7\eta}{6} + 2\log_n\left(\max_{0\leq i < n}\left(B\left(i,m_P\right)\right)2^{\frac{\eta}{12}}\right) + 2$, then
\[
\left|a_{PF}(n)\right| \leq n^{\frac{\ell}{2}-1}.
\]
As these 5 cases are exhaustive, letting
\[
M\left(n,m_P,m_F\right) = \begin{cases} L(n,m_F,1) &\mbox{if } m_P = 0 \\
2\log_n\left(B(n,m_P)\right)+2 & \mbox{if } m_F = 0\\
 2\log_{\frac{n}{n-1}}\left(6\max_{1\leq i < n}\left(B\left(i,m_P\right)\right)\right)& \mbox{if } m_F = 1\\
\frac{7\eta}{6} + 2\log_n\left(\max_{0\leq i < n}\left(B\left(i,m_P\right)\right)2^{\frac{\eta}{12}}\right) + 2& \mbox{otherwise}
 \end{cases} 
\]
is sufficient to have $\left|a_{PF}(n)\right| \leq n^{\frac{\ell}{2}-1}$.
\end{proof}
Now we consider the equation
\begin{equation*}
\sum_{i=1}^{r}A_{i}P_{i}F_{i}=h .
\end{equation*}
To prove the main theorem, fix the variables as in Theorem~\ref{thm:equivmain} and let $\ell$ be the weight of $h$ and each term in the polynomial. 
Now fix an odd prime $p$ such that $p > \left(\sum_{i=0}^r\left|A_i\right|\right)^2 + \sum_{i=0}^r\left|A_i\right|$.
As $h$ is an eigenform, it obeys the $p^{th}$ order Hecke relation
\[
p^{\ell-1}=\left(a_h(p)\right)^2-a_h\left(p^2\right).
\]
Using \eqref{eq:cuspidal-poly} and the triangle inequality, we get
\[
p^{\ell-1} \leq \left(\sum_{i=0}^r\left|A_i\right|\left|a_{P_iF_i}(p)\right|\right)^2 + \sum_{i=0}^r\left|A_i\right|\left|a_{P_iF_i}\left(p^2\right)\right|.
\]
If $\ell > \max_{1\leq i\leq r}\left(\max\left(M\left(p,m_{P_i},m_{F_i}\right),M\left(p^2,m_{P_i},m_{F_i}\right)\right)\right)$, then each term meets the requirements of Proposition~\ref{prop:cuspidalcaseassist}, giving
\begin{align*}
p^{\ell-1} & \leq \left(\sum_{i=0}^r\left|A_i\right|\left|a_{P_iF_i}(p)\right|\right)^2 + \sum_{i=0}^r\left|A_i\right|\left|a_{P_iF_i}\left(p^2\right)\right| \\
& \leq \left(\sum_{i=0}^r\left|A_i\right|p^{\frac{\ell}{2}-1}\right)^2 + \sum_{i=0}^r \left|A_i\right|p^{\ell-2} \\
& = p^{\ell-2}\left(\left(\sum_{i=0}^r\left|A_i\right|\right)^2 + \sum_{i=0}^r\left|A_i\right|\right) \\
& < p^{\ell-1}.
\end{align*}
However, this is a contradiction, so $\ell \leq \max_{1\leq i\leq r}\left(\max\left(M\left(p,m_{P_i},m_{F_i}\right),M\left(p^2,m_{P_i},m_{F_i}\right)\right)\right)$. This bounds the weight of $h$, which is sufficient to prove Theorem~\ref{thm:equivmain} in this case.

\section{Examples of eigenform polynomial relations}\label{section:examples}
In this section we give some explicit examples of polynomial decompositions of eigenforms for $\SL_2(\Z)$. We remark that it is in general easy to derive examples of polynomial decompositions by finding linearly dependent sets of products of eigenforms, which is easy to do because the space of modular forms over $\SL_2(\Z)$ for a fixed weight is finite-dimensional.

We begin with the well-known identity for $E_{12}$:
%
%
%
\begin{equation} \label{eq:E12_relation}
E_{12} =  \frac{441}{691} E_4^3 + \frac{250}{691}E_6^2.
\end{equation}


If $\M_k(\SL_2(\Z))$ has dimension 2 and $\Delta_{k}$ denotes the unique cuspidal eigenform of $\M_k(\SL_2(\Z))$, we obtain for certain constants $A_1,A_2$ that

\begin{equation} \label{eq:f24_relation}
A_1E_{6} \Delta_{18} + A_2E_4 \Delta_{20} = g_1 \approx q - 0.0001962q^{2} + 195660.0094185q^{3} + O(q^{4}) .
\end{equation}

\begin{equation} \label{eq:f24_relation2}
B_1E_{6} \Delta_{18} + B_2E_4 \Delta_{20} = g_2 \approx q + 0.0002489q^{2} + 195659.9880488q^{3} +  O(q^{4}) .
\end{equation}

Here $g_1$ and $g_2$ are the cuspidal eigenforms of $M_{24}\left(\SL_2(\Z)  \right)$, which is a 3-dimensional space. It may be shown with elementary linear algebra that the $A_i $ and $B_i$ are the algebraic numbers satisfying $A_1 > B_1$, $A_2 < B_2$, and the following minimal polynomials:

\[
A_1, B_1: x^{2} - \frac{131909627}{163749888} x + \frac{9915382466495}{61119318196224}=0
\]
\[
 A_2, B_2: x^{2} - \frac{195590149}{163749888} x + \frac{21799696204223}{61119318196224}=0.
\]

These identities may be verified by checking equality of sufficiently (but finitely) many Fourier coefficients due to the Sturm bound \cite{sturm_bound}.


It is interesting to note that our methods allow one to compute explicit upper bounds on weight in the cuspidal case. A simple procedure determines a value $\nu$ for a given polynomial $P$, such that if a cuspidal eigenform $f$ may be written as the output of a polynomial $P$ with eigenform inputs, then the weight of $f$ may be at most $\nu$. This was implemented in code available from the second author, and a few examples are shown below. Note that $A_1,A_2,B_1,$ and $B_2$ are the same as above.

\begin{table*}[!h]
\begin{center}
\begin{tabular}[b]{|c|c|}
\hline
Polynomial & $\nu$-bound \\
\hline
$P(x,y) = \frac{441}{691} x^3 + \frac{250}{ 691} y^2$ & 66\\
\hline
$P(x,y,z,w) = A_1 xy + A_2 zw$ &  988\\
\hline
$P(x,y,z,w) = B_1 xy + B_2 zw$ &  988\\
\hline
\end{tabular}
\end{center}
\caption{Weight bounds in the cuspidal case for various polynomials}
\end{table*}

\section{Acknowledgments}\label{sec:ack}
The authors are undergraduates at the University of Michigan and wrote this paper as part of the University of Michigan REU program. The authors would like to thank the program for hosting and funding the research. The authors are grateful to their advisor, Dr. Benjamin Linowitz of the University of Michigan, for his valuable input over the course of the summer, and for reading numerous drafts of the paper. Additionally, the authors would like to thank Dr. Matthew Stover of Temple University for his help arranging the project.


\bibliographystyle{plain}

\end{document}